\documentclass[12pt,twoside]{article} 
\usepackage{geometry}
\usepackage{amsmath,amsthm}
\usepackage{xcolor}
\usepackage{color}
\usepackage{graphicx}
\usepackage{verbatim}
\usepackage{amssymb}
\usepackage{epstopdf}

\date{} 
\newtheorem{theorem}{Theorem}[section]

\newtheorem{example}{Example}[section]

\begin{document} 
\centerline{\bf{Collective marks and first passage times} }
\centerline{}

\centerline{\bf {Yiping ZHANG}\textsuperscript{a}} 

\centerline{\bf {Myron HLYNKA}\textsuperscript{a}} 

\centerline{\bf {Percy H. BRILL}\textsuperscript{a,b}} 

\centerline{} 
\centerline { Department of Mathematics and Statistics\textsuperscript{a}} 
\centerline{ School of Business\textsuperscript{b}}
\centerline{University of Windsor} 

\centerline{Windsor, Ontario, Canada N9B 3P4}

\centerline{}

\begin{abstract}Probability generating functions for first passage times of Markov chains are found using the method of collective marks. A system of equations is found which can be used to obtain moments of the first passage times. 
\end{abstract} 

{\bf AMS Subject Classification:} 60J05, 60J22 \\ 

{\bf Keywords:} Markov chains, first passage times, collective marks
\section{Introduction} 

Suppose we have a Markov chain with $n$ states labeled $1,2,\dots, n$. 

Define the random variable $X_{ij}$ to be the number of steps needed to move from state $i$ to state $j$ for the first time. We refer to $X_{ij}$ as the first passage time. Define the first passage probability as $f_{ij}(k)=P(X_{ij}=k)$. 
There are several ways to compute the first passage probabilities. For example, see Hunter (\cite{Hun}) and Kao (\cite{Kao}). 
First passage probabilities are important as they can be used to control processes and determine when to implement parameter changes. 

Suppose we have a probability mass function for a discrete random variable $X$ that takes on value $k$ with probability $p_k$ for $k=0,1,\dots$. Define the probability generating function for $X$ to be $\psi_{X}(z)=\sum_{k=0}^{\infty}p_k z^k$. Alfa (\cite{Alf}, p. 76) gives an expression for the probability generating function of the first passage probabilities from state $i$ to state $j$ as follows.
\begin{equation*}
\psi_{ij}(z)=\dfrac{P_{ij}(z)}{1-P_{ij}(z)}
\end{equation*}
where $P_{ij}(z)=\sum_{k=1}^{\infty}p_{ij}^{(k)}z^k$. But this is not a closed form since we need the values $p_{ij}^{(k)}$. 

The method of collective marks was originated by van Dantzig (\cite{Van}), and discussed in Runnenburg (\cite{Run}) and Kleinrock (\cite{Kle}, chapter 7). The method gives a probabilistic interpretation of a probability generating function $\sum_{k=0}^{\infty}p_k z^k$. Let $z$ be the probability that an item is ``marked.'' Then $p_k z^k$ represents the probability that random variable $X$ takes on the value $k$ and each of the $k$ counts is marked. Summing over all $k$ gives the total probability that all items from a single realization of the random variable $X$ are marked. 

In this paper, we use the collective marks method to find the probability generating function for first passage probabilities, in a closed form for a fixed number of states $n$. We find expressions for moments of the first passage times. We present a method to find probability generating functions of second passage times. 
 
\section{Computing first passage probabilities}

\begin{theorem}
\label{A}
Let $\psi_{ij}(z)$ be the probability generating function for the first passage random variable from $i$ to $j$ for an $n$ state Markov chain.  Then we obtain an equation,
\begin{equation*}
\psi_{ij}(z)=p_{ij}z+\sum_{k:k\neq j}p_{ik}z\psi_{kj}(z)
\end{equation*}
\end{theorem} 
\begin{proof}
By the method of collective marks, $\psi_{ij}(z)$ represents the probability that the path starting from $i$ and reaching $j$ for the first time has all of its steps receiving a mark. Here the probability of a step being marked is assumed to be $z$. The first step may enter state $j$ immediately and this occurs with probability $p_{ij}$. The probability that the singleton path is marked is $z$. So $p_{ij}z$ is the probability that the first passage probability consists of 1 step and is marked. Otherwise, the process goes to some other state $k$ with probability $p_{ik}$ and that step is marked with probability $z$. From the new position $k$, the process moves to state $j$ eventually with each step being marked with probability generating function $\psi_{kj}(z)$. Summing over all cases gives the result. 
\end{proof}

\noindent
Note: The equation in our theorem involves the  generating functions $\psi_{kj}(z)$ (for all $k$) and we can get a similar equation for each of these. For fixed $j$, this will give us a linear system of equations  in the variables $\psi_{1j}(z),\dots,\psi_{nj}(z)$, which can be solved to get any particular first passage generating function desired as a non linear function of $z$. The coefficients in the system of equations may involve $z$ as well as constants. 

\begin{theorem}
\label{B} 
Let $\psi_{13}(z)$ be the probability generating function for the first passage random variable from $1$ to $3$ for an $3$ state Markov chain.  Then 
\begin{equation*}
\psi_{13}(z)=\dfrac{p_{13}z+(p_{12}p_{23}-p_{13}p_{22})z^2}
{1-(p_{11}+p_{22})z+(p_{11}p_{22}-p_{12}p_{21})z^2}
\end{equation*}
\end{theorem} 
\begin{proof}
From Theorem \ref{A}, we have
\begin{align*}
\psi_{13}(z)&=p_{11}z\psi_{13}(z)+p_{12}z\psi_{23}(z)+p_{13}z\\
\psi_{23}(z)&=p_{21}z\psi_{13}(z)+p_{22}z\psi_{23}(z)+p_{23}z
\end{align*}
 Solving this system of two equations in two unknowns gives our result. 
\end{proof}

\noindent
Note: \\
(a) A similar result holds for any pair, not just $(i,j)$.\\
(b) Our method manages to obtain a closed form for the probability generating function of the first passage  times for 3 state Markov chains\\
(c) Theorem \ref{B} can be extended to a larger number number of states as we still essentially get a linear system to solve. \\
(d) Although the system of equations is linear in the $\psi_{ij}(z)$ unknowns, the coefficients involve the variable $z$, the the resulting expressions are nonlinear functions of $z$.   

\section{Example}
\begin{example}\label{C}
Consider  the Markov transition matrix
$P=\begin{bmatrix}
.2 &.4&.4\\
.3&.3&.4\\
.5&.4&.1
\end{bmatrix}$
We will compute first passage probability generating functions for $\psi_{13}(z)$, $\psi_{23}(z)$ , and $\psi_{33}(z)$.
For the first two we use theorem \ref{B} (with appropriate changes for $\psi_{23}(z)$, and for the third, we get a separate equation.  
\end{example}
According to Theorem \ref{B}, the probability generating function for the first passage probabilities from state 1 to state 3 is given by 
\begin{equation*}
\psi_{13}(z)=\dfrac{.4z+(.4*.4-.4*.3)z^2}{1-(.2+.3)z+(.2*.3-.4*.3)z^2}= \dfrac{.4z+.04z^2}{1-.5z-.06z^2}
\end{equation*}  
We use the Maple command \\
$series({\frac { 0.4\,z+ 0.04\,{z}^2}{1- 0.5\,z- 0.06\,z^2}},z,8)$\\
to find the Taylor expansion and get results.\\
\begin{equation*}
\psi_{13}(z)= 0.4z+ 0.24z^2+ 0.144z^3+ 0.0864z^4+ 0.05184z^5+ 0.031104z^6+ 0.0186624z^7+ \dots
\end{equation*}
This result agrees with other methods. \\
In a similar manner, we find 
\begin{equation*}
\psi_{23}(z)=\dfrac{.4z+(.3*.4-.4*.2)z^2}{1-(.3+.2)z+(.3*.2-.3*.4)z^2}= \dfrac{.4z+.04z^2}{1-.5z-.06z^2}
\end{equation*}  
Finally,
\begin{align*}
\psi_{33}(z)&=p_{33}z+p_{31}z\psi_{13}(z) +\psi_{32}z\psi_{23}(z) =.1z+.5\psi_{13}(z)+.4\psi_{23}(z)\\
&= \dfrac{.1z-.05z^2-.006z^3+.2z^2+.02z^3+.16z^2+.016z^3}{1-.5z-.06z^2}=\dfrac{.1z+.31z^2+.03z^3}{1-.5z-.06z^2}
\end{align*}  

\section{Moments of first passage times}

Theorem \ref{B} gives an expression for $\psi_{ij}(z)$ so we can find the moments of the first passage probabilities by simply taking derivatives and evaluating the expressions at $z=1$, making any additional computations needed. But this explicitly requires solving for   $\psi_{ij}(z)$ which can be a somewhat burdensome task as the coefficients of the linear system involve the variable $z$.

Theorem \ref{A} gives an equation for $\psi_{ij}(z)$ involving the probability generating function of first passage times from $i$ to $j$ and  since we have similar expressions for $\psi_{kj}(z)$ (for $k\neq j$), we have a system of equations that  we can work with.  We can take the derivative of the SYSTEM of equations, and then substitute $z=1$ into the system to create a much more tractible system of equations. Of course, $\psi_{ij}(1)=1$ and $\psi_{ij}^{\prime}(1)=\mu_{ij}$ where $\mu_{ij}=E(X_{ij}$, where $X_{ij}$ is the number of steps needed to reach state $j$ from state $i$ for the first time.  Also, $\psi_{ij}^{(2)}(1)=E(X_{ij}(X_{ij}-1)).$ 
\begin{example}
We use the same $3 \times 3$ transition matrix as in Example \ref{C}
\end{example}
The system of equations from Theorem \ref{A} is
\begin{align*}
\psi_{13}(z)&=.2z\psi_{13}(z)+.4z\psi_{23}(z)+.4z\\
\psi_{23}(z)&=.3z\psi_{13}(z)+.3z\psi_{23}(z)+.4z
\end{align*}

Taking derivatives gives
\begin{align*}
\psi_{13}^{\prime}(z)&=.2\psi_{13}(z)+.2z\psi_{13}^{\prime}(z)+.4\psi_{23}(z)+.4z\psi_{23}^{\prime}(z)+.4\\
\psi_{23}^{\prime}(z)&=.3\psi_{13}(z)+.3z\psi_{13}^{\prime}(z)+.3\psi_{23}(z)+.3z\psi_{23}^{\prime}(z)+.4
\end{align*}

Evaluating at $z=1$ gives
\begin{align*}
\mu_{13}=.2+.2\mu_{13}+.4+.4\mu_{23}+.4=1+.2\mu_{13}+.4\mu_{23}\\
\mu_{23}=.3+.3\mu_{13}+.3+.3\mu_{23}+.4=1+.3\mu_{13}+.3\mu_{23}\\
\end{align*}
Solving these gives $\mu_{13}=2.5$ and $\mu_{23}=2.5$.

\section{Second passage times}

\begin{theorem}
Let $Y_{ij}$ be the random variable representing the number of steps needed to move from $i$ to $j$ for the second time. 
Then the probability generating function for $Y_{ij}$ is $\psi_{ij}(z) \psi_{jj}(z)$
\end{theorem}
\begin{proof}
$Y_{ij}=X_{ij}+X_{jj}$ where $X_{ij}$ is the first passage random variable, so $Y_{ij}$ is just the convolution of two independent random variables.
Since the pgf of a convolution is the product of the pgf's of each part, the result follows. 
\end{proof} 

\begin{example}
We will compute the second passage time from state 1 to state 3 in the Markov chain with transition matrix $P=\begin{bmatrix}
.2 &.4&.4\\
.3&.3&.4\\
.5&.4&.1
\end{bmatrix}$
We earlier calculated \\
$\psi_{13}(z)=\dfrac{.4z+.04z^2}{1-.5z-.06z^2}$ and $\psi_{33}(z)=\dfrac{.1z+.31z^2+.03z^3}{1-.5z-.06z^2}$ so \\
$\psi_{second}(z)=\dfrac{( .4z+.04z^2)(.1z+.31z^2+.03z^3)}{(1-.5z-.06z^2)^2}$.
If we expand this (using MAPLE) into a Taylor series, we get\\
$\psi_{second}(z)=0.04z^{2}+ 0.168z^{3}+ 0.1872z^{4}+ 0.16416z^{5}+\dots$\\
Thus, for example,  the probability of moving from 1 to 3 for the second time on step 4 is 0.1872.
\end{example}
In a similar manner,we can obtain higher order passage probabilities. 
\bigskip

\noindent
{\bf Acknowledgments.} 
We acknowledge funding and support  from NSERC (Natural Sciences and Engineering Reseach Council of Canada).


\begin{thebibliography}{9} 
\bibitem{Alf}
A.S. Alfa.  \emph{Applied Discrete-time Queues}, second edition. Springer. 2014


\bibitem{Hun}
J.J. Hunter, \emph{Mathematical Techniques of Applied Probability Vol. 1}. Academic Press. 1983

\bibitem{Kao}
E. Kao.  \emph{An Introduction to Stochastic Processes}. Duxbury Press. 1996

\bibitem{Kle}
L. Kleinrock. \emph{Queueing Systems, Volume 1}. Chapter 7. Wiley. 1975

\bibitem{Run} J.T. Runnenburg. On the use of Collective Marks in Queueing Theory. In W.L. Smith and W.E. Wilkinson, editors, Congestion Theory. pp. 399-438. University of North Carloina Press. 1965. 


\bibitem{Van} 
D. Van Dantzig. Sur methode des fonctions generatrices. Colloques internationaux du CNRS, 13: 29-45, 1949




\end{thebibliography}
\end{document}